\pdfoutput=1
\documentclass[english]{jnsao}
\usepackage[utf8]{inputenc}
\usepackage{bm,amsthm,amsmath,eucal,enumitem,mathshortcuts,mathtools}
\usepackage{algorithmic,algorithm}
\usepackage{graphicx,hyperref,cite,tikz}
\definecolor{poly1}{RGB}{100,100,200}
\definecolor{poly1dark}{RGB}{40,40,180}
\definecolor{poly2}{RGB}{240,160,160}
\definecolor{poly2dark}{RGB}{180,40,40}

\usepackage[nameinlink,capitalize]{cleveref}

\newtheorem{assumption}[theorem]{Assumption}

\title{Constructing {a subgradient} from directional derivatives for
  functions of two variables}
\shorttitle{Constructing a subgradient from directional derivatives}

\author{Kamil A. Khan\thanks{Corresponding author.\newline K.A.~Khan and
		Y.~Yuan are with the Department of Chemical Engineering, McMaster University,
		Hamilton, ON, Canada. Email:~kamilkhan@mcmaster.ca. } \and Yingwei Yuan}
\shortauthor{K.A.~Khan and Y.~Yuan}

\acknowledgements{
  This work was supported by the Natural Sciences and Engineering
Research Council of Canada (NSERC) under Grant RGPIN-2017-05944. 
}

\manuscriptcopyright{{\copyright} the authors}
\manuscriptlicense{CC-BY-SA 4.0}

\manuscriptsubmitted{2020-01-30}

\manuscriptaccepted{2020-06-04}

\manuscriptvolume{1}

\manuscriptyear{2020}

\manuscriptnumber{6061}

\manuscriptdoi{10.46298/jnsao-2020-6061}

\manuscripteprinttype{arXiv}
\manuscripteprint{2001.10670}

\begin{document}

\maketitle

\begin{abstract}
      For any {scalar-valued} bivariate  function that is locally Lipschitz continuous and
    directionally differentiable, it is shown that {a
    	subgradient} may
    always be constructed from the function's directional derivatives in
    the four compass directions, arranged in a so-called ``compass
    difference''.  When the original function is nonconvex, the obtained
    subgradient is an element of Clarke's generalized gradient, but the
    result appears to be novel even for convex functions. The
    function is not required to be represented in any particular form,
    and no further assumptions are required, though the result is
    strengthened when the function is additionally L-smooth in the sense
    of Nesterov. For certain optimal-value functions and certain
    parametric solutions of differential equation systems, these new
    results appear to provide the only known way to compute
    {a subgradient}. These results also imply that centered finite
    differences {will converge to a subgradient} for bivariate nonsmooth
    functions. As a dual result, we find that any compact convex
    set in two dimensions contains the midpoint of its interval
    hull. Examples are included for illustration, and it is {demonstrated} that
    these results do not extend directly to functions of more than two
    variables or sets in higher dimensions.
\end{abstract}


\section{Introduction}
\label{sec:introduction}

Subgradient methods~\cite{Shor,nonsmoothBook} and bundle
methods~\cite{Kiwiel,Bundle2,Luksan,nonsmoothBook} for nonsmooth
optimization {typically use a subgradient at each iteration} to provide local sensitivity information
that is {ultimately} useful enough to {infer} descent. For convex problems, these
subgradients are elements of the convex subdifferential; for nonconvex
problems, the subgradients must typically be elements of either
Clarke's generalized gradient~\cite{Clarke} or other established
generalized
subdifferentials~\cite{Nesterov,Mordukhovich,Kruger}. Evaluating
{a subgradient} directly, however, may be a challenging task; this
difficulty has motivated the development of numerous subdifferential
approximations~\cite{nonsmoothBook,Facchinei,PangStewart}.

Nevertheless, there are several settings in which evaluating
directional derivatives is much simpler than evaluating {a subgradient}
using established methods. For finite compositions of simple smooth
and nonsmooth functions, directional derivatives may be evaluated
efficiently~\cite{GriewankNonsmooth} by extending the standard
forward/tangent mode of algorithmic differentiation~\cite{Griewank},
while extensions to efficient subgradient evaluation methods require
more care~\cite{KhanBartonFwdAD,KhanBranch}. Directional derivatives
of implicit functions and inverse functions may be obtained by solving
auxiliary equation systems~\cite{Scholtes}, whereas subgradient
results in this setting {assume either} special
structure~\cite{Scholtes,KhanBartonHybrid} or a series of recursive
equation-solves~\cite{KhanBartonHybrid}.  For solutions of parametric
ordinary differential equations (ODEs) with nonsmooth right-hand
sides, directional derivatives may be evaluated by solving an
auxiliary ODE system~\cite[Theorem~7]{PangStewart} using a standard
ODE solver, whereas the only general method for subgradient evaluation
involves solving a series of ODEs punctuated by discrete jumps that
must be handled carefully~\cite{KhanBartonPlenary,KhanThesis}.  In
parametric optimization, Danskin's classical
result~\cite{Danskin,Hogan} describes directional derivatives for
optimal-value functions as the solutions of related optimization
problems in a general setting, while subgradient results such as
\cite[Theorem~5.1]{StechlinskiNLP} {tend to additionally} require unique solutions
for the embedded optimization problem.

Moreover, directional derivatives and subdifferentials of convex
functions are essentially duals~\cite{HiriartUrrutyLemarechal}. Hence,
this article examines the question of whether, given a
directional-derivative evaluation oracle for a function and little
else, this oracle may be used to compute {a
	subgradient at each iteration of a typical nonsmooth optimization method}. This is clearly true for univariate
functions, for example; in this case, the entire subdifferential may
be constructed from directional derivatives in the positive
and negative directions.

To address this question, this article defines a function's \emph{compass differences} to be
vectors obtained by arranging directional derivatives in the
coordinate directions and negative coordinate directions in a certain
way. Thus, for a bivariate function, a compass difference involves
directional derivatives in the four compass directions. For a bivariate
function that is locally Lipschitz continuous and directionally
differentiable, it is shown  that {the compass difference at
	any domain point is
	a subgradient}, with {this subgradient} understood to
be {an element of Clarke's generalized gradient} in the nonconvex case. Surprisingly, while this result
is 
simple to state, it appears to be previously unknown even for convex
functions, and does not require any additional assumptions. It is also
shown that this result does not extend directly to functions of more
than two variables.  As a related
result, this article shows that a compact convex set in $\reals^2$ must
always contain the midpoint of its interval hull, though this does not
extend directly to sets in $\reals^n$ for $n>2$. Hence, four calls to
a directional-derivative evaluation oracle are sufficient to compute a
subgradient for a nonsmooth bivariate function, and centered finite
differences for these functions are useful approximations of
{a subgradient}. In several cases, the approach of this article appears
to be  the only way known thus far to evaluate {a subgradient} correctly.

Audet and Hare~\cite{AudetHare} studied a similar problem involving
the similar setup, in the field of geometric
probing~\cite{Skiena}. {Unlike our work, Audet and Hare additionally
	assume that: (a) their oracle $\mathsf{D}$ is convex (as a
	set's support function), (b) the bivariate function's regular
	subdifferential is polyhedral, and (c) the 
	oracle $\mathsf{D}$ evaluates the function's directional derivative\footnote{{We
			note briefly
			that the piecewise-linear
			function $f:x\in\reals^2\mapsto\max(0,\min(x_1,x_2))$ has a
			polyhedral regular subdifferential at $\bar{x}\coloneqq (0,0)$, but here the
			oracle $\mathsf{D}$ of [1] is distinct from both the directional
			derivative $d\mapsto f'(\bar{x};d)$
			(c.f. \cref{def:dirDeriv} below) and Clarke's generalized
			directional derivative $d\mapsto f^\circ(\bar{x};d)$
			(from~\cite{Clarke}). In this case $d\mapsto f'(\bar{x};d)$ is
			nonconvex, so (a) and (c) cannot be satisfied simultaneously even if
			$\mathsf{D}$ is redefined.}}.  These
	assumptions are evidently satisfied, for example, by any function that
	is both convex and piecewise-differentiable in the sense of Scholtes~\cite{Scholtes}. Under these
	assumptions, Audet and Hare present} a method to use finitely many
{directional derivative evaluations} to construct the whole {regular} subdifferential at a
given domain point. This method proceeds by
deducing each vertex of the subdifferential, and depends heavily on
the assumption of subdifferential polyhedrality; its complexity scales
linearly with the number of subdifferential vertices. It is readily
verified, for example, that their {Algorithm~1} will run
forever without locating any subgradients when applied to the convex
Euclidean norm function:
\[
f:\reals^2\to\reals: x\mapsto \sqrt{x_1^2+x_2^2}
\]
at $x=0$. Indeed, their algorithm is not intended to work in this case.
Unlike the work of~\cite{AudetHare}, we do
not assume that subdifferentials are polyhedral, { do not
	require the subdifferential's support function to be available in
	the nonconvex case, and do not
	assume directional derivatives to be convex with respect to direction}. Our goal is 
only to identify one subgradient rather than {a} whole subdifferential;
characterizing {a} whole subdifferential in closed form may be difficult or
impossible when we {do not} know \emph{a priori} that it is
polyhedral. As mentioned above, in the nonconvex case, we
evaluate an element of Clarke's generalized
gradient~\cite{Clarke} instead of a subgradient.

The remainder of this article is structured as follows.
\cref{sec:background} summarizes relevant established
constructions in nonsmooth analysis, \cref{sec:subgradient}
defines compass differences in terms of directional derivatives  and shows that they are valid
subgradients, and \cref{sec:examples} presents several examples
for illustration.

\section{Mathematical background}
\label{sec:background}

The Euclidean norm $\|\cdot\|$ and inner product
$\innerProd{\cdot}{\cdot}$ are used throughout this article. The $\supth{i}$ unit
coordinate vector in $\reals^n$ is denoted as
$\vCoord{i}$, and components of vectors are indicated using
subscripts, e.g.~$x_i\coloneqq \innerProd{\vCoord{i}}{x}$.
The convex hull and the closure of a set $S\subset\reals^n$ are
denoted as $\convHull{S}$ and $\closure{S}$, respectively.

\subsection{Directional derivatives and convex subgradients}

\begin{definition}
	\label{def:dirDeriv}
	Consider an open set $X\subset\reals^n$ and a function
	$f:X\to\reals$. The following limit, if it exists, is the
	\emph{(one-sided) directional derivative} of $f$ at $x\in X$ in the
	direction $d\in\reals^n$:
	\[
	\dd{f}{x}{d} \coloneqq  \lim_{t\downTo{0}}\frac{f(x+td) - f(x)}{t}.
	\]
	If $\dd{f}{x}{d}$ exists in $\reals$ for each $d\in\reals^n$, then
	$f$ is \emph{directionally differentiable} at $x$.
\end{definition}

This article primarily considers situations where directional
derivatives are available via a black-box oracle. For example, this oracle could
represent symbolic calculation, the situation-specific
directional derivatives described in
\cref{sec:introduction}, algorithmic
differentiation~\cite{Griewank}, or even finite difference
approximation if some error is tolerable.

The primary goal of this article is to use directional derivatives to
evaluate {a subgradient}, defined for convex functions as follows, and
generalized to nonconvex functions as in \cref{sec:nonsmooth}
below. {Individual subgradients} are used {at each
	iteration of} subgradient methods for convex
minimization~\cite{Shor} and bundle methods for nonconvex
minimization~\cite{HiriartUrrutyLemarechal2}. They are also used to
build useful affine outer approximations for nonconvex
sets~\cite{Rote,KhanSubtangent}. In each of these applications, only a
single subgradient is needed at each {visited} domain point.

\begin{definition}
	Given a convex set $X\subset\reals^n$ and a convex function
	$f:X\to\reals$, $s\in\reals^n$ is a \emph{subgradient} of $f$ at
	$x\in X$ if
	\begin{equation}
	\label{eq:subtangent}
	f(y)\geq f(x) + \innerProd{s}{y-x},
	\qquad\forall y\in X.
	\end{equation}
	The set of all subgradients of $f$ at $x$ is the
	\emph{(convex) subdifferential} $\genJac{f}(x)$.
\end{definition}
In this definition, if $X$ is open, then $\genJac{f}(x)$ is convex, compact, and
nonempty~\cite{HiriartUrrutyLemarechal}. The directional derivative and
subdifferentials of $f$ at $x$ are related as follows~\cite{HiriartUrrutyLemarechal}. For each $d\in\reals^n$,
\begin{equation}
\label{eq:ddAndSub}
\dd{f}{x}{d} = \max\{\innerProd{s}{d}:\,\,s\in\genJac{f}(x)\}
\quad\text{and}\quad
\genJac{f}(x) = \{s\in\reals^n:
\dd{f}{x}{d}\geq\innerProd{s}{d},\,\,\forall d\in\reals^n\}.
\end{equation}
Thus, the subdifferential characterizes the local behavior of
convex functions via~\eqref{eq:ddAndSub}, and characterizes the global
behavior of convex functions via~\eqref{eq:subtangent}. Moreover,
\eqref{eq:ddAndSub} shows that
directional derivatives and subgradients of convex functions are
essentially duals of each other.

\subsection{Nonsmooth analysis}
\label{sec:nonsmooth}

The following constructions by Clarke~\cite{Clarke} extend certain subgradient
properties to nonconvex functions, and are used in methods for
equation-solving~\cite{Qi,LPNewton} and
optimization~\cite{Shor,Kiwiel,MakelaBundle}.

\begin{definition}
	\label{def:Clarke}
	Consider an open set $X\subset\reals^n$ and a locally Lipschitz
	continuous function $f:X\to\reals$. The \emph{(Clarke--)generalized
		directional derivative} of $f$ at $x\in X$ in the direction
	$d\in\reals^n$ is:
	\[
	f^\circ(x;d)
	\coloneqq \limsup_{\substack{y\to x \\ t\downTo{0}}}\frac{f(y+td) - f(y)}{t}.
	\]
	\emph{Clarke's generalized gradient} of $f$ at $x$ is then:
	\[
	\genJac{f}(x)\coloneqq 
	\{s\in\reals^n: 
	f^\circ(x;d)\geq\innerProd{s}{d}, \quad\forall d\in\reals^n\}.
	\]
	Elements of Clarke's generalized gradient will be called
	\emph{Clarke subgradients}.
\end{definition}

{With $f$ as in the above definition, and for any $x\in X$, $\genJac{f}(x)$} is guaranteed to be nonempty, convex,
and compact in $\reals^n$. As suggested by its notation, Clarke's generalized gradient does
indeed coincide with the convex subdifferential when $f$ is
convex~\cite{Clarke}. When $f$ is nonconvex, \eqref{eq:ddAndSub} is
no longer guaranteed to hold with Clarke's generalized gradient in place of
the convex subdifferential. The following result for univariate
functions is easily demonstrated, {and is summarized in~\cite{AudetHare}}.

\begin{proposition}
	Consider an open set $X\subset\reals$ and a univariate function
	$f:X\to\reals$ that is locally
	Lipschitz continuous and directionally differentiable. For each $x\in X$,
	\begin{equation}
	\label{eq:ddUnivariate}
	\genJac{f}(x)=\convHull\{\dd{f}{x}{1},-\dd{f}{x}{-1}\}.
	\end{equation}
\end{proposition}

Hence, one call to an oracle that evaluates directional derivatives is
sufficient to obtain a single Clarke subgradient for such a univariate function
$f$. It will be shown in this article that, for bivariate functions $f:\reals^2\to\reals$
that are locally Lipcshitz continuous and directionally differentiable, four directional derivative
evaluations are sufficient to evaluate a single Clarke
subgradient. 

The following definition by Nesterov~{\cite{Nesterov}} will be used to specialize this
result in a useful way. Nesterov's definition is based on repeated directional
differentiation, and permits certain extensions of calculus
rules for smooth functions to nonsmooth functions.
\begin{definition}
	Consider an open set $X\subset\reals^n$ and a locally Lipschitz
	continuous function
	$f:X\to\reals$. The function $f$ is \emph{lexicographically
		(L--)smooth} at $x\in X$ if the following conditions are satisfied:
	\begin{itemize}
		\item $f$ is directionally differentiable at $x$, 
		\item with $f^{(0)}\coloneqq \dd{f}{x}{\cdot}$, for any collection of vectors
		$m_{(1)},\ldots,m_{(n)}\in\reals^n$, the following inductive sequence of
		higher-order directional derivatives is well-defined:
		\[
		f^{(k)}\coloneqq  \dd{[f^{(k-1)}]}{m_{(k)}}{\cdot},
		\qquad\text{for each }k\in\{1,2,\ldots,n\}.
		\]
	\end{itemize}
	If these vectors $m_{(i)}$ are linearly independent, then $f^{(n)}$ is linear, and its constant gradient
	is called a \emph{lexicographic subgradient} of $f$ at $x$. The
	\emph{lexicographic subdifferential} $\lexDiff{f}(x)$ is the set of
	all lexicographic subgradients of $f$ at $x$.
\end{definition}

All convex functions on open domains in $\reals^n$ are
L-smooth~\cite{Nesterov}, as are differentiable functions, functions
that are piecewise differentiable in the sense of
Scholtes~\cite{Scholtes,KhanBartonFwdAD}, and functions that are
well-defined finite compositions of other
L-smooth functions~\cite{Nesterov}. Further characterizations of
L-smoothness have been developed for certain optimal-value
functions~\cite{StechlinskiNLP}, and for parametric systems of
ordinary differential equations or differential-algebraic
equations~\cite{KhanBartonPlenary,StechlinskiDAE}.

\section{Constructing {a subgradient} from directional derivatives}
\label{sec:subgradient}

This section defines \emph{compass differences} for functions in terms of
directional derivatives, and shows that {a compass difference of a
	bivariate function
	is a subgradient}. As a
corollary, it is also shown that any compact convex  set in $\reals^2$
contains the midpoint of its interval hull. As there is nothing
particularly special about the compass directions in this context,
other choices of directions are also considered.

\subsection{Compass differences}

\begin{definition}
	Consider an open set $X\subset\reals^n$ and a function
	$f:X\to\reals$ that is directionally differentiable at $x\in X$. The
	\emph{compass difference} of $f$ at $x$ is a vector
	$\dcentre{f}(x)\coloneqq (\dcentrei{1}{f}(x),\ldots,\dcentrei{n}{f}(x))\in\reals^n$ for which, for each $i\in\{1,\ldots,n\}$,
	\[
	{\dcentrei{i}{f}(x)}\coloneqq \frac{1}{2}(\dd{f}{x}{\vCoord{i}} - \dd{f}{x}{-\vCoord{i}}).
	\]
\end{definition}

The compass difference is so named because it considers how $f$
behaves when its argument is varied
in each of the compass directions. This metaphor works best when
$n=2$; this case is also the focus of this article.

Evaluating
$\dcentre{f}(x)$ ostensibly requires $2n$ directional derivative evaluations. However, if directional derivative values are not available,
compass differences may instead be approximated using finite differences.
Observe that the compass difference of a function is a centered finite
difference of the directional derivative mapping $\dd{f}{x}{\cdot}$
at $0$. From the definition of the directional derivative, we have,
for each $i\in\{1,\ldots,n\}$,
\[
{\dcentrei{i}{f}(x)}
= \lim_{\delta\downTo{0}}
\frac{f(x+\delta\vCoord{i}) - f(x-\delta\vCoord{i})}{2\delta}.
\]
So, if numerical evaluations of $f:\reals^n\to\reals$ are viable but evaluations of
$\dd{f}{x}{\cdot}$ are not, then $2n$ evaluations of $f$ may be used to
approximate $\dcentre{f}(x)$ using the argument of the above
limit. That is,
for sufficiently small $\delta>0$,
\begin{equation}
\label{eq:finiteDifference}
\dcentre{f}(x)
\approx \frac{1}{2\delta}
\begin{bmatrix}
f(x+\delta\vCoord{1}) - f(x-\delta\vCoord{1}) \\
f(x+\delta\vCoord{2}) - f(x-\delta\vCoord{2}) \\
\vdots \\
f(x+\delta\vCoord{n}) - f(x-\delta\vCoord{n})
\end{bmatrix},
\end{equation}
{which is incidentally the centered simplex gradient of $f$ at $x$ with a
	sampling set comprising the coordinate vectors (c.f.~\cite{ConnScheinbergVicente})}.
However, if $f$ is evaluated here using a numerical method, and if $\delta$
is too small, then the subtraction operations in this
approximation may introduce unacceptable numerical error. This
drawback is typical of finite difference approximations.


\subsection{Nonconvex functions of two variables}
\label{sec:nonconvex}

{As in~\cite{Facchinei}, let us say that a function $\reals^n\to\reals$ is \emph{B-differentiable} if it is
	both directionally differentiable and locally Lipschitz continuous.
	This section presents the main result of this article: that any compass
	difference of a B-differentiable function of two variables is a
	Clarke subgradient.} This result is strengthened somewhat when the
considered {function is L-smooth, and is also
	specialized to convex functions and convex sets in the subsequent sections.}

To our
knowledge, {the main result in this section} is the first general {closed-form} description of a Clarke
subgradient for a nonconvex {bivariate}
function in terms of that function's directional
derivatives {(in the sense of    \cref{def:dirDeriv})}. Moreover, the result shows that four calls to a
directional derivative oracle are sufficient to evaluate a
Clarke subgradient for a bivariate {B-differentiable} function, without any
further structural knowledge of the function at all. Unlike
established characterizations of generalized subgradients such as
\cite[Proposition~4.3.1]{Scholtes} and \cite[Theorem~3.5]{KhanBranch}, this result does not require $f$
to be represented in any particular format.

The following mean-value theorem will be useful
in this development.

\begin{lemma}
	\label{lem:meanValueLip}
	Consider a function $\psi:\reals^n\to\reals$ that is positively
	homogeneous and locally Lipschitz continuous.
	For any $x,y\in\reals^n$, there exists $s\in\genJac{\psi}(0)$ for
	which
	\begin{equation}
	\label{eq:meanValueLip}
	\psi(y) - \psi(x) = \innerProd{s}{y-x}.
	\end{equation}
	If $\psi$ is also L-smooth, then there exists
	$s\in\convHull{\lexDiff{\psi}(0)}\subset \genJac{\psi}(0)$ satisfying \eqref{eq:meanValueLip}.
\end{lemma}
\begin{proof}
	We first proceed without the L-smoothness assumption. According to
	Lebourg's mean-value theorem \cite[Theorem~2.3.7]{Clarke}, the
	equation \eqref{eq:meanValueLip} holds for some
	$z\in\convHull{\{x,y\}}$ and some $s\in\genJac{\psi}(z)$. Since
	$\psi$ is positively homogeneous and locally Lipschitz continuous, \cite[Lemma~3.1]{KhanBartonPlenary}
	implies that $\genJac{\psi}(z)\subset\genJac{\psi}(0)$, and so
	$s\in\genJac{\psi}(0)$, as required.
	
	Next, if $\psi$ is additionally assumed to be L-smooth, then the
	final claimed result is obtained by a similar argument, applying
	Nesterov's mean-value theorem \cite[Theorem~12]{Nesterov} instead of
	Lebourg's, applying \cite[Lemma~4]{KhanDirected} instead of
	\cite[Lemma~3.1]{KhanBartonPlenary}, and applying
	\cite[Theorem~11]{Nesterov} to establish the inclusion
	{$\convHull{\lexDiff{\psi}(0)}\subset\genJac{\psi}(0)$}.
\end{proof}


The following theorem is the main result of this article, and rests
heavily on \cref{lem:meanValueLip}. It shows that {any
	compass difference of a B-differentiable function is a Clarke subgradient}, and specializes this result to
L-smooth functions.

\begin{theorem}
	\label{thm:nonconvex}
	Consider an open set $X\subset\reals^2$ and a locally Lipschitz
	continuous function $f:X\to\reals$. If $f$ is directionally
	differentiable at some $x\in X$, then $\dcentre{f}(x)\in\genJac{f}(x)$.  Moreover, if $f$ is L-smooth at $x\in X$, then $\dcentre{f}(x)\in\clConvHull{\lexDiff{f}(x)}\subset\genJac{f}(x)$.
\end{theorem}
\begin{proof}
	Suppose that $f$ is directionally differentiable at $x\in X$.  Consider the auxiliary mapping:
	\[
	\psi:y\mapsto \dd{f}{x}{y}-\innerProd{\dcentre{f}(x)}{y},
	\]
	and observe that $\psi$ is Lipschitz continuous~\cite{Scholtes}, and that
	$\dd{f}{x}{y}=\psi(y)+\innerProd{\dcentre{f}(x)}{y}$ for each
	$y\in\reals^2$. Thus, Clarke's calculus rule for addition
	\cite[Corollary~1 to Proposition~2.3.3]{Clarke} implies:
	\[
	\genJac{[\dd{f}{x}{\cdot}]}(0)
	= \{a + \dcentre{f}(x): a\in\genJac{\psi}(0)\}.
	\]
	Moreover, \cite[Corollary~3.1]{KhanBartonPlenary} and \cite{Imbert}
	imply $\genJac{[\dd{f}{x}{\cdot}]}(0)\subset\genJac{f}(x)$, and so
	\[
	\{a + \dcentre{f}(x): a\in\genJac{\psi}(0)\} \subset \genJac{f}(x).
	\]
	It therefore suffices to show that $0\in\genJac{\psi}(0)$.
	
	{Now, observe that $\psi$ is positively homogeneous, and so
		$\psi$ is equivalent to $\dd{\psi}{0}{\cdot}$. Thus, for each
		$i\in\{1,2\}$,
		\begin{align*}
			\dcentrei{i}{\psi}(0)
			&= \tfrac{1}{2}(\psi(e_{(i)})-\psi(-e_{(i)})) \\
			&= \tfrac{1}{2}\left[
			(\dd{f}{x}{e_{(i)}}-\innerProd{\dcentre{f}(x)}{e_{(i)}})
			- (\dd{f}{x}{-e_{(i)}}-\innerProd{\dcentre{f}(x)}{-e_{(i)}})
			\right] \\
			&= \dcentrei{i}{f}(x) - \innerProd{\dcentre{f}(x)}{e_{(i)}} \\
			&= 0.
		\end{align*}
		Hence $\dcentre{\psi}(0)=0$.
	}

	To obtain a contradiction, suppose that
	$0\notin\genJac{\psi}(0)$. Then, since $\genJac{\psi}(0)$ is convex
	and closed, there must exist a
	strictly separating hyperplane between $0$ and
	$\genJac{\psi}(0)$. That is, there exist a nonzero vector
	$p\coloneqq (p_1,p_2)\in\reals^2$ and a scalar $a>0$ for
	which
	$\innerProd{p}{s}\geq a$ for each $s\in
	\genJac{\psi}(0)$.
	
	Since $\dcentre{\psi}(0)=0$, we have
	$\psi(1,0)=\psi(-1,0)$ and $\psi(0,1)=\psi(0,-1)$.
	Since $\psi$ is positively homogeneous, we then have
	$\psi(p_1,0)=\psi(-p_1,0)$ and $\psi(0,p_2)=\psi(0,-p_2)$
	(regardless of the signs of $p_1$ and $p_2$). Subtraction then
	yields:
	\begin{equation}
	\label{eq:struts}
	\psi(p_1,0) - \psi(0,-p_2) = \psi(-p_1,0) - \psi(0,p_2).
	\end{equation}
	
	Now, according to \cref{lem:meanValueLip}, there exist vectors
	$\eta,\sigma\in \genJac{\psi}(0)$ for which
	\[
	\psi(p_1,0) - \psi(0,-p_2) = \innerProd{p}{\eta}
	\qquad\text{and}\qquad
	\psi(-p_1,0) - \psi(0,p_2) = -\innerProd{p}{\sigma}.
	\]
	Hence, since $\innerProd{p}{s}\geq a$ for each $s\in
	\genJac{\psi}(0)$, we have
	\[
	\psi(p_1,0) - \psi(0,-p_2)\geq a > 0 > -a \geq \psi(-p_1,0) - \psi(0,p_2),
	\]
	which contradicts \eqref{eq:struts}. Thus,
	$0\in  \genJac{\psi}(0)$ as required.    
	
	Next, suppose that $f$ is L-smooth at $x\in X$. The inclusion $\lexDiff{f}(x)\subset\genJac{f}(x)$ was shown by
	Nesterov \cite[Theorem~11]{Nesterov}; since $\genJac{f}(x)$ is
	closed and convex, it follows that
	$\clConvHull{\lexDiff{f}(x)}\subset\genJac{f}(x)$.
	{Consider the auxiliary mapping $\psi$ as above, and note that
		\eqref{eq:struts} still holds.
		The calculus rules of the lexicographic
		subdifferential~\cite[Theorem~5 and Definitions~1 and~5]{Nesterov}
		imply that both $\dd{f}{x}{\cdot}$ and $\psi$ are L-smooth at $0$, and that
		\[
		\lexDiff{f}(x)=\lexDiff{[\dd{f}{x}{\cdot}]}(0)
		=\{a+\dcentre{f}(x):a\in\lexDiff{\psi}(0)\}.
		\]
		From here, a similar argument to the previous case shows
		that $\dcentre{f}(x)\in\clConvHull{\lexDiff{f}(x)}$.}
\end{proof}



{Intuitively,} there is nothing special about the coordinate directions used to
construct a compass difference, and a change of
basis in \cref{thm:nonconvex} may be carried out as follows.
\begin{corollary}
	\label{cor:nonconvex}
	Consider an open set $X\subset\reals^2$, a locally Lipschitz
	continuous function
	$f:X\to\reals$, and a nonsingular matrix $V\in \reals^{2\times 2}$.
	If $f$ is directionally differentiable at some $x\in X$, and if $v_{(i)}$ denotes the $\supth{i}$ column of
	$V$, then
	\[
	\frac{1}{2}(\transpose{V})^{-1}
	\begin{bmatrix}
	\dd{f}{x}{v_{(1)}} - \dd{f}{x}{-v_{(1)}} \\
	\dd{f}{x}{v_{(2)}} - \dd{f}{x}{-v_{(2)}} 
	\end{bmatrix}
	\in \genJac{f}(x).
	\]
\end{corollary}
\begin{proof}
	Consider auxiliary mappings:
	\[
	g:\reals^2\to\reals^2: y\mapsto x + V(y-x)
	\]
	and $h\coloneqq f\circ g$. Let
	\[
	z\coloneqq \frac{1}{2}\begin{bmatrix}
	\dd{f}{x}{v_{(1)}} - \dd{f}{x}{-v_{(1)}} \\
	\dd{f}{x}{v_{(2)}} - \dd{f}{x}{-v_{(2)}} 
	\end{bmatrix}.
	\]
	The chain rule for directional
	derivatives~\cite[Theorem~3.1.1]{Scholtes} implies that
	$\dcentre{h}(x)=z$, and so \cref{thm:nonconvex} shows that $z\in\genJac{h}(x)$.
	Since $V$ is nonsingular, $g$ is surjective, in which case
	\cite[Theorem~2.3.10]{Clarke} implies that:
	\[
	\genJac{h}(x) = \{\transpose{V}s: s\in\genJac{f}(x)\}.
	\]
	Thus, $z=\transpose{V}s$ for some $s\in\genJac{f}(x)$, and so
	$(\transpose{V})^{-1}z\in\genJac{f}(x)$ as claimed.
\end{proof}

{The particular Clarke subgradients identified by
	\cref{thm:nonconvex} and \cref{cor:nonconvex} do not
	necessarily coincide.}

We may remove the directional differentiability requirement of
\cref{thm:nonconvex} as follows, by employing {Clarke's}
generalized directional
derivative $f^\circ$ from \cref{def:Clarke}. We note, however, that the generalized directional derivative
is typically inaccessible in practice.

\begin{corollary}
	Given an open set $X\subset\reals^2$, a locally Lipschitz continuous function
	$f:X\to\reals$, and some $x\in X$,
	\[
	\dcentre{[f^\circ(x;\cdot)]}(0) = \frac{1}{2}
	\begin{bmatrix}
	f^\circ(x;(1,0)) - f^\circ(x;(-1,0)) \\
	f^\circ(x;(0,1)) - f^\circ(x;(0,-1))
	\end{bmatrix}
	\in \genJac{[f^\circ(x;\cdot)]}(0)=\genJac{f}(x).
	\]
\end{corollary}
\begin{proof}
	Established results
	\cite[Section~V, Proposition~2.1.2]{HiriartUrrutyLemarechal} and
	\cite[Proposition~2.1.2]{Clarke} imply that $f^\circ(x;\cdot)$ is
	convex and positively homogeneous (as the support function of
	$\genJac{f}(x)$), and has the subdifferential $\genJac{f}(x)$ at
	$0$. {Moreover, as a convex function on an open domain, $f^\circ(x;\cdot)$ is locally
		Lipschitz continuous and directionally differentiable~\cite{HiriartUrrutyLemarechal}. Hence, \cref{thm:nonconvex}} implies the claimed result.
\end{proof}

\subsection{Convex functions of two variables}
\label{sec:convex}

{This section specializes \cref{thm:nonconvex}} to convex functions; this specialization
appears to be a novel result in convex analysis and is simpler to
state. Namely, {any compass
	difference of a bivariate convex function is in fact a subgradient in the
	traditional sense}. Hence, four directional derivative evaluations are
sufficient to construct a subgradient of a bivariate convex function.

\begin{corollary}
	\label{cor:compassConvex}
	Consider an open convex set $X\subset\reals^2$ and a convex function
	$f:X\to\reals$. For each $x\in X$, $\dcentre{f}(x)\in\genJac{f}(x)$.
\end{corollary}
\begin{proof}
	Since $f$ is convex and $X$ is open, $f$ is
	locally Lipschitz continuous and directionally
	differentiable~\cite{HiriartUrrutyLemarechal}. {The claimed result
		then follows immediately from
		\cref{thm:nonconvex}.}
\end{proof}

\begin{corollary}
	\label{cor:convex}
	Consider an open convex set $X\subset\reals^2$, a convex function
	$f:X\to\reals$, and a nonsingular matrix $V\in \reals^{2\times 2}$.
	For any $x\in X$, with $v_{(i)}$ denoting the $\supth{i}$ column of
	$V$,
	\[
	\frac{1}{2}(\transpose{V})^{-1}
	\begin{bmatrix}
	\dd{f}{x}{v_{(1)}} - \dd{f}{x}{-v_{(1)}} \\
	\dd{f}{x}{v_{(2)}} - \dd{f}{x}{-v_{(2)}} 
	\end{bmatrix}
	\in \genJac{f}(x).
	\]
\end{corollary}
\begin{proof}
	Again, since $f$ is locally Lipschitz continuous and directionally differentiable, the claimed corollary is a special case of 
	\cref{cor:nonconvex}.
\end{proof}

\subsection{Compact convex sets in $\reals^2$}

This section applies \cref{cor:compassConvex} to show that
any nonempty compact convex set in $\reals^2$ contains the center of
its smallest enclosing box (or \emph{interval}). These notions are formalized in the
following classical definitions {(summarized in~\cite{Neumaier})}, followed by the claimed result.

\begin{definition}
	An \emph{interval} in $\reals^n$ is a nonempty set of the form
	$\{x\in\reals^n:a\leq x\leq b\}$, where $a,b\in\reals^n$, and where each inequality is to be
	interpreted componentwise. The \emph{midpoint} of an interval
	$\{x\in\reals^n:a\leq x\leq b\}$ is $\frac{1}{2}(a+b)\in\reals^n$.

	Given a bounded set $B\subset\reals^n$, the \emph{interval hull} of
	$B$ is the intersection in $\reals^n$ of all interval supersets of
	$B$. 
\end{definition}

The interval hull of a bounded set $B\subset\reals^n$ is itself an
interval, and is, intuitively, the smallest interval superset of $B$. Support functions of convex sets, defined as follows and discussed at
length in~\cite{HiriartUrrutyLemarechal}, are useful when relating convex
sets to properties of subdifferentials of convex functions.

\begin{definition}
	Given a set $C\subset\reals^n$, the \emph{support function}
	of $C$ is the mapping:
	\[
	\sigma_C:\reals^n\to\reals\cup\{\pm\infty\}:
	d\mapsto\sup\{\innerProd{d}{x}:x\in C\}.
	\]
\end{definition}

The following corollary uses support functions to extend
\cref{cor:compassConvex} to the problem of locating an element
of a closed convex set in $\reals^2$.

\begin{corollary}
	\label{cor:support}
	Any nonempty compact convex set $C\subset\reals^2$ contains the
	midpoint of its interval hull. 
\end{corollary}
\begin{proof}
	The interval hull of $C$ may be expressed in terms of the support
	function $\sigma_C$ as:
	\[
	\{x\in\reals^2:
	\quad -\sigma_C(-1,0)\leq x_1\leq \sigma_C(1,0),
	\quad -\sigma_C(0,-1)\leq x_2\leq \sigma_C(0,1)\};
	\]
	the midpoint of this interval hull is then
	\[
	z\coloneqq  \frac{1}{2}
	\begin{bmatrix}
	\sigma_C(1,0) - \sigma_C(-1,0) \\
	\sigma_C(0,1) - \sigma_C(0,-1)
	\end{bmatrix}.
	\]
	As shown in \cite[Section~VI, Example~3.1]{HiriartUrrutyLemarechal},
	$\sigma_C$ is directionally differentiable at $0$, with
	$\dd{(\sigma_C)}{0}{d}=\sigma_C(d)$ for each $d\in\reals^2$. Thus,
	$\dcentre{\sigma_C}(0)=z$.
	
	Next, \cite[Section~VI, Example~3.1]{HiriartUrrutyLemarechal} also
	shows that $\sigma_C$ is convex, with
	$\genJac{\sigma_C}(0)=C$. Combining these observations with 
	\cref{cor:compassConvex} yields $z=\dcentre{\sigma_C}(0)\in \genJac{\sigma_C}(0)=C$,
	as claimed.
\end{proof}

\section{Examples}
\label{sec:examples}

This section illustrates the main results of this
article.
\cref{sec:counterexamples} motivates the assumptions of
\cref{cor:compassConvex} and \cref{cor:support} by showing how these results could fail if their
assumptions were weakened.
\cref{sec:applications} {uses} compass differences to compute
{individual subgradients} in cases where this was previously difficult or impossible.

\subsection{Counterexamples for related claims}
\label{sec:counterexamples}

The following example shows that, for functions mapping
$\reals^2$ into $\reals$, compass {differences} are not necessarily
elements of either {the regular
	subdifferential~\cite{RockafellarWets},} the lexicographic subdifferential~\cite{Nesterov}, the B-subdifferential~\cite{Scholtes,QiConvergence},
or the Mordukhovich upper subdifferential~\cite{Mordukhovich}.

\begin{example}
	Consider the concave piecewise-linear function:
	\[
	f:\reals^2\to\reals:x\mapsto -|x_1|.
	\]
	Direct computation yields $\dcentre{f}(0,0)=(0,0)$, which is indeed
	an element of $\genJac{f}(0,0)=\{(\lambda,0):-1\leq\lambda\leq
	1\}$. However, the lexicographic subdifferential, the
	B-subdifferential, and the Mordukhovich upper subdifferential of $f$
	at $(0,0)$ are each equal to $\{(-1,0),(1,0)\}$, which does not
	contain $(0,0)$. {The regular subdifferential of $f$ at
		$(0,0)$ is empty.}
\end{example}

The following example shows that {\cref{thm:nonconvex},
	\cref{cor:compassConvex}, and \cref{cor:support}} are minimal in the sense
that, under the respective assumptions of these results, three support
function evaluations are {generally} not sufficient to infer a set element, and
three directional derivative evaluations are generally not
sufficient to infer a function's subgradient. 

\begin{example}
	Suppose that $C\subset\reals^2$ is the unit ball
	$\{x\in\reals^2:\|x\|\leq 1\}$, which has the constant support function
	$\sigma_C:d\mapsto 1$. Consider three nonzero points $u,v,w\in\reals^2$ in
	general position. From the support function's definition, if we
	{did not know the set} $C$ but did know that
	$\sigma_C(u)=\sigma_C(v)=\sigma_C(w)=1$, then we could infer that $C$ is
	a subset of the triangle:
	\[
	T\coloneqq \{x\in\reals^2:\innerProd{u}{x}\leq 1, \quad\innerProd{v}{x}\leq
	1,\quad\innerProd{w}{x}\leq 1\}.
	\]
	Denote the three vertices of $T$ as $a,b,c\in\reals^2$, and denote
	the three edges of $T$ as
	\[
	T_1\coloneqq \convHull\{a,b\},
	\qquad T_2\coloneqq \convHull\{b,c\},
	\qquad\text{and}\quad
	T_3\coloneqq \convHull\{a,c\}.
	\]
	Since $\{a,b\}\subset T_1\subset T$, observe that
	\[
	\sigma_C(u)=1=\sigma_T(u)\geq\sigma_{T_1}(u) \geq \innerProd{u}{x}\qquad\forall x\in\{a,b\}.
	\]
	But, since $a,b,c$ are the vertices of the triangle $T$, and since
	one edge of $T$ lies on the line
	$\innerProd{u}{x}=1$, it cannot be  that
	$\innerProd{u}{a}$ and $\innerProd{u}{b}$ are both less than
	1. Hence $\sigma_{T_1}(u)\geq 1$, and so $\sigma_C(u)=\sigma_{T_1}(u)$.
	
	Similar logic shows that $\sigma_{T_i}(u)=\sigma_C(u)$, $\sigma_{T_i}(v)=\sigma_C(v)$,
	and $\sigma_{T_i}(w)=\sigma_C(w)$ for each
	$i\in\{1,2,3\}$. Each $T_i$ is compact and convex, and the intersection $T_1\cap T_2\cap T_3$ is
	empty.  Hence, there is no way to infer an element of $C$ from
	the support function evaluations $\sigma_C(u)$, $\sigma_C(v)$, and
	$\sigma_C(w)$ and the knowledge that $C$ is compact and
	convex; these support function evaluations are consistent with the
	incorrect hypotheses $C=T_1$, $C=T_2$, and $C=T_3$, yet these guesses
	have no point in common.
	
	Similarly, considering the convex Euclidean norm function
	\[
	f:\reals^2\to\reals:x\mapsto \|x\|,
	\]
	it is readily verified that $\genJac{f}(0)=C$. Suppose we know
	nothing about $f$ other than its convexity and the fact that
	$\dd{f}{0}{u}=\dd{f}{0}{v}=\dd{f}{0}{w}=1$. In this case, there is no way to
	infer an element of $\genJac{f}(0)$ from these three directional
	derivatives alone, since for each $i\in\{1,2,3\}$, the functions
	\[
	\phi_i:\reals^2\to\reals:d\mapsto\max\{\innerProd{s}{d}:s\in T_i\}
	\]
	all have the same directional derivatives as $f$ at $0$ in  the
	directions $u$, $v$, and $w$. However, their subdifferentials at $0$
	are the sets $T_i$, which have no point in common.
\end{example}

The following example shows that the results of this article do not
extend directly to functions of more than two variables or sets in
more than two dimensions.

\begin{example}
	\label{ex:no3d}
	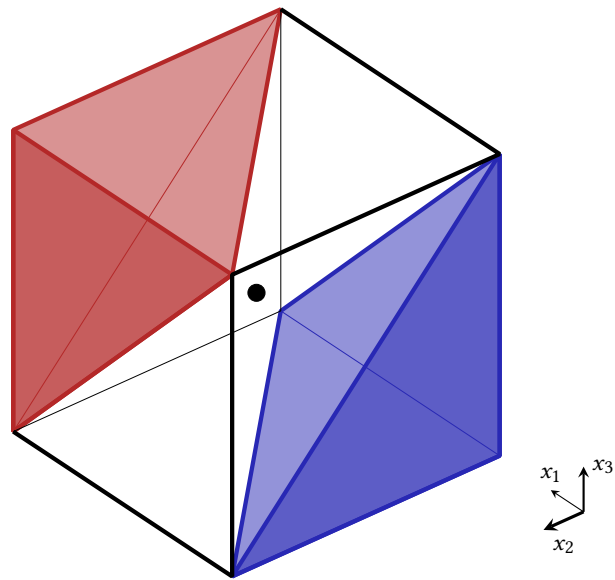
\begin{figure}[t]
		\centering
		\footnotesize
		\begin{tikzpicture}[scale=0.8,
		x={(-1.8cm,1.2cm)}, y={(-2.2cm,-1cm)}, z={(0cm,2.5cm)}]
		\coordinate (A1) at (-1,-1,-1);
		\coordinate (A2) at (1,-1,-1);
		\coordinate (A3) at (-1,1,-1);
		\coordinate (A4) at (-1,-1,1);
		\coordinate (B1) at (1,1,1);
		\coordinate (B2) at (1,1,-1);
		\coordinate (B3) at (1,-1,1);
		\coordinate (B4) at (-1,1,1);
		\coordinate (originL) at (-1.4,-1.3,-1.3);
		
		
		\draw [color=poly1dark, line join=bevel, ultra thick, fill opacity=0.5, fill=poly1dark]
		(A1) -- (A3) -- (A2) -- (A4) -- cycle;
		\draw [color=poly1dark, line join=bevel, ultra thick, fill opacity=0.5, fill=poly1dark]
		(A1) -- (A3) -- (A4) -- cycle;
		\draw [color=poly1dark, very thin] (A1) -- (A2);
		
		\draw [color=poly2dark, line join=bevel, ultra thick, fill opacity=0.5, fill=poly2dark] (B1) -- (B3) -- (B4) -- (B2) -- cycle; 
		\draw [color=poly2dark, line join=bevel, ultra thick, fill opacity=0.5, fill=poly2dark] (B1) -- (B4) -- (B2) --
		cycle;
		\draw [color=poly2dark, very thin] (B2) -- (B3);
		
		\draw [ultra thick, line join=bevel] (B3) -- (A4) --  (B4) -- (A3) -- (B2);
		\draw [very thin, line join=bevel] (B3) -- (A2) -- (B2);
		
		\filldraw (0,0,0) circle (4pt);
		
		\draw [->, >=stealth, thin] (originL) -- ++(0.3,0,0) node[above]
		{$x_1$};
		\draw [->, >=stealth, very thick] (originL) -- ++(0,0.3,0) node[below right] {$x_2$};
		\draw [->, >=stealth, thick] (originL) -- ++(0,0,0.3) node[right] {$x_3$};
		\end{tikzpicture}
		\caption{The disjoint convex compact sets $C_1$ (red) and $C_2$
			(blue) in $\reals^3$ described in \cref{ex:no3d}, and the common midpoint
			(black dot)
			of their interval hulls.}
		\label{fig:3d}
	\end{figure}

	Consider the following convex compact sets in $\reals^3$:
	\begin{align*}
		C_1 &\coloneqq  \convHull\{(1,1,-1),(-1,1,1),(1,-1,1),(1,1,1)\}, \\
		\text{and}\quad
		C_2 &\coloneqq  \convHull\{(1,-1,-1),(-1,1,-1),(-1,-1,1),(-1,-1,-1)\}.
	\end{align*}
	These sets are illustrated in \cref{fig:3d}. They are
	disjoint; for any $x\in C_1$ and $y\in C_2$, and  with
	$e\coloneqq (1,1,1)\in\reals^3$, observe that
	\[
	\innerProd{e}{x} \geq 1 > -1 \geq \innerProd{e}{y}.
	\]
	However, it is readily verified that both $C_1$ and $C_2$ have the
	interval hull $[-1,1]^3$, whose midpoint is $(0,0,0)$, which is in
	neither $C_1$ nor $C_2$. Thus, \cref{cor:support} does not
	extend immediately to $\reals^3$.
	
	Similarly, consider the following two convex piecewise-linear functions:
	\begin{align*}
		f:\reals^3\to\reals: 
		x&\mapsto \max\{x_1 + x_2 - x_3,\quad  x_2 + x_3 - x_1,\quad x_3 + x_1 -
		x_2\}, \\
		\phi:\reals^3\to\reals: 
		x&\mapsto \max\{x_1 - x_2 - x_3, \quad x_2 - x_3 - x_1,\quad  x_3 - x_1 -
		x_2\}.
	\end{align*}
	According to \cite[Proposition~4.3.1]{Scholtes},
	$\genJac{f}(0)\subset\convHull\{(1,1,-1),(-1,1,1),(1,-1,1)\}\subset C_1$, and
	$\genJac{\phi}(0)\subset\convHull\{(1,-1,-1),(-1,1,-1),(-1,-1,1)\}\subset
	C_2$. Thus, the subdifferentials $\genJac{f}(0)$ and
	$\genJac{\phi}(0)$ are disjoint.
	Moreover, it is readily verified that:
	\[
	1 = \dd{f}{0}{s\vCoord{i}} = \dd{\phi}{0}{s\vCoord{i}},
	\qquad\forall s\in\{-1,+1\},
	\quad\forall i \in\{1,2,3\}.
	\]
	Thus, the functions $f$ and $\phi$ cannot
	be distinguished based on their directional derivatives at $0$ in
	any coordinate direction or negative coordinate
	direction, and  $\dcentre{f}(0)=\dcentre{\phi}(0)=0$, but the two functions'
	subdifferentials at $0$ are disjoint. This shows that
	\cref{thm:nonconvex} and \cref{cor:compassConvex} do not extend
	immediately to functions of three variables.
\end{example}

The following example illustrates that the assumption in
\cref{cor:support} that $C$ is closed is crucial.

\begin{example}
	\label{ex:noOpen}
	
	
	
	
	
	Consider the convex set:
	\[
	C\coloneqq \{x\in\reals^2: -1<x_1, \quad x_2<1, \quad x_1<x_2\}.
	\]
	Observe that $C$ is not closed, and that the interval hull of $C$ is
	$[-1,1]^2$. The midpoint of this hull is $(0,0)$, which is
	not an element of~$C$.
\end{example}

\subsection{Applications}
\label{sec:applications}

\subsubsection{Solutions of parametric differential equations}

This section applies \cref{thm:nonconvex} to describe correct
{single subgradients} for solutions of parametric ordinary differential equations
(ODEs) with parameters in $\reals^2$. This
approach reduces to the classical ODE sensitivity approach of
\cite[Section~V, Theorem~3.1]{Hartman1964} when the original ODE is
defined in terms of smooth functions. Unlike existing methods~\cite{KhanBartonPlenary} for
generalized derivative evaluation for these systems, the approach of
this article describes {a subgradient} in terms of auxiliary ODE systems
that can be integrated numerically using off-the-shelf ODE solvers,
but is of course restricted to systems with two parameters.

We consider the following setup, which is readily adapted to other ODE representations.

\begin{assumption} \label{asu1}
	Consider functions $f:\reals^n\to\reals^n$,
	$x_0:\reals^2\to\reals^n$, and $g:\reals^2\times\reals^n\to\reals$ that are locally Lipschitz continuous and
	directionally differentiable. For some scalar $t_f>0$, let
	$x:[0,t_f]\times \reals^2$ be defined so that, for each
	$p\in\reals^2$, $x(\cdot,p)$ solves the following ODE system uniquely:
	
	\[
	\frac{dx}{dt}(t,p)=f(x(t,p)),\quad x(0,p)=x_0(p).
	\]
	
	Define $\phi:\reals^2\to\reals$ to be the cost function:
	\[
	\phi:p \mapsto g(p,x(t_f,p)).
	\]
	
\end{assumption}

Under this assumption, {a subgradient} for $\phi$ may be computed by
combining the results of this article with directional derivatives
described by \cite[Theorem~7]{PangStewart} as follows. If it is desired for the
ODE right-hand-side to depend explicitly on $t$, then an alternative
directional derivative result \cite[Theorem~4.1]{KhanBartonPlenary}
may be used instead.

\begin{proposition}\label{theor:1}
	Suppose  that \cref{asu1} holds, and consider some
	particular  $p\in\reals^2$. For each $d\in\reals^2$,
	let $y(\cdot,d)$ denote a solution on $[0,t_f]$ of the following ODE:
	\begin{align}\label{eq:y}
		\frac{dy}{dt}(t,d)=f'(x(t,p);y(t,d))\quad y(0,d)=x_0'(p;d).
	\end{align}
	{Then $y(\cdot,d)$ is in fact the unique solution of this ODE} for each $d\in\reals^2$. Moreover, if
	we define
	\[
	\psi(d)\coloneqq \dd{g}{(p,x(t_f,p))}{(d,y(t_f,d))}
	\]
	for each
	$d\in\reals^2$, then
	\[
	\frac{1}{2}
	\begin{bmatrix}
	\psi(1,0)-\psi(-1,0) \\ \psi(0,1) - \psi(0,-1)
	\end{bmatrix}
	\]
	is an element of $\genJac{\phi}(p)$. 
\end{proposition}

\begin{proof}
	According to \cite[Theorem~7]{PangStewart}, $y(t,d)$ is the
	directional derivative $\dd{x}{(t,p)}{(0,d)}$ for each $t\in[0,t_f]$
	and $d\in\reals^2$. The result then follows immediately from \cref{cor:nonconvex}
	and the chain rule \cite[Theorem~3.1.1]{Scholtes}.
\end{proof}

If lexicographic derivatives are unavailable for the functions in
\cref{asu1} or do not exist, then \cref{theor:1} is,
to our knowledge, the first method for describing a subgradient of
$\phi$. The following numerical example illustrates this proposition.

\begin{example}
	\label{ex:ode}

	Consider  a function
	$x_0:\reals^2\to\reals^3:p\mapsto(p_1,p_2,p_1)$. For each
	$p\in\reals^2$, let $x(\cdot,p)$ denote the unique solution on
	$[0,1]$ of {the} following parametric ODE system. Here dotted variables
	denote derivatives with respect to $t$.
	\begin{align*}
		\dot{x}_1  &= |x_1|+|x_2|+x_3, \\
		\dot{x}_2&=|x_2|, \\
		\dot{x}_3&=x_3,\\
		x(0,p)&=x_0(p).
	\end{align*}
	Consider a cost function $\phi:p\mapsto x_1(1,p)$. In this case, for
	each $d\in\reals^2$ the ODE~\eqref{eq:y} becomes:
	\begin{align*}
		\dot{y}_1 &=\begin{cases}
			-y_1-y_2+y_3, & \text{if} \ x_1 < 0, \ x_2 < 0,\\
			-y_1+y_2+y_3, & \text{if} \ x_1 < 0, \ x_2 > 0,\\
			-y_1+|y_2|+y_3, & \text{if} \ x_1 < 0, \ x_2 = 0,\\
			y_1-y_2+y_3, & \text{if} \ x_1 > 0, \ x_2 < 0,\\
			y_1+y_2+y_3, & \text{if} \ x_1 > 0, \ x_2 > 0,\\
			y_1+|y_2|+y_3, & \text{if} \ x_1 > 0, \ x_2 = 0,\\
			|y_1|-y_2+y_3, & \text{if} \ x_1 = 0, \ x_2 < 0,\\
			|y_1|+y_2+y_3, & \text{if} \ x_1 = 0, \ x_2 > 0,\\
			|y_1|+|y_2|+y_3, & \text{if} \ x_1 = 0, \ x_2 = 0,\\
		\end{cases} \\
		\dot{y}_2&=\begin{cases}
			y_2, & \text{if} \ x_2 > 0,\\
			|y_2|, & \text{if} \ x_2=0,\\
			-y_2, & \text{if} \ x_2 < 0,\\
		\end{cases} \\
		\dot{y}_3 &=y_3,
	\end{align*}
	with $y(0,d)\equiv(d_1,d_2,d_1)$. In this case $\phi$ is convex. {To
		evaluate a compass difference of $\phi$, the
		numerical variable-step variable-order ODE solver \texttt{ode15s} was used in
		\textsc{Matlab} to evaluate $y$
		numerically, using \textsc{Matlab}'s default precision (on the
		order of 16 significant digits) for arithmetic, and using respective local
		absolute and relative tolerances of $10^{-6}$ and $10^{-3}$ for each
		integration step. Thus, to within the corresponding computational
		error, we obtained $\dcentre{\phi}(0)\approx (3.490,0.772)=:s$, and 
		\cref{theor:1} yields $\dcentre{\phi}(0)\in\genJac{\phi}(0)$. \cref{fig:ode}
		shows that $s$ does indeed appear to satisfy~\eqref{eq:subtangent},
		and does thereby appear to be a subgradient of $\phi$ at $0$ to within
		numerical precision.}	
            \end{example}

            	\begin{figure}[t]
		\centering
		\includegraphics[width=0.5\textwidth]{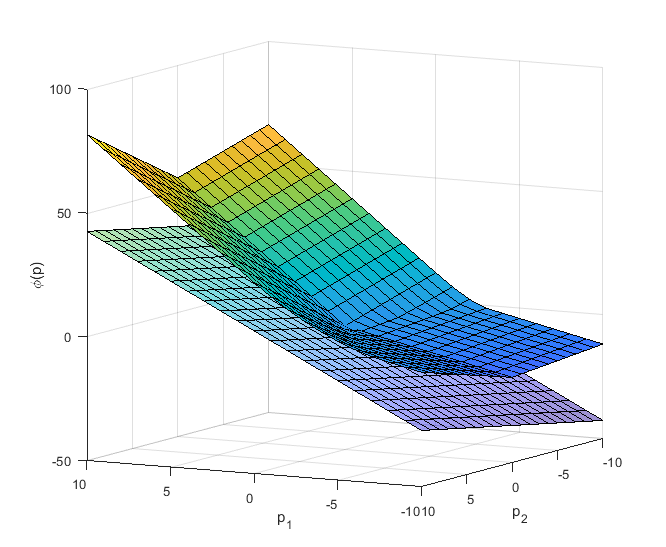}
		\caption{Plot of the function $\phi$ (top) described in
			\cref{ex:ode}, which {appears to dominate} the approximation
			$p\mapsto \phi(0)+\innerProd{s}{p}$ (bottom) based on the computed
			{compass difference $s$ of $\phi$ at~$0$}.}
		\label{fig:ode}
	\end{figure}

\subsubsection{Optimal-value functions}

A well-known result by Danskin \cite[Theorem~1]{Danskin} describes
directional derivatives for certain optimal-value functions, and has
been extended to a variety of settings
(e.g.~\cite{Hogan,BertsekasBook}). {The following proposition
	and its proof are} intended to show how any of
these results may be combined with \cref{thm:nonconvex} or \cref{cor:compassConvex} 
to describe {a subgradient in
	each case}.

\begin{proposition}
	Consider a compact set $C\subset\reals^n$, some open superset $Z$ of
	$C$, and a continuously differentiable function $f:\reals^2\times Z\to \reals$. Define an
	optimal-value function $\phi:\reals^2\to\reals$ for which
	\[
	\phi:x\mapsto\min\{f(x,y):y\in C\}.
	\]
	For some particular $\hat{x}\in\reals^2$, define the following:
	\begin{itemize}
		\item a set $Y\coloneqq \{\hat{y}\in C: f(\hat{x},\hat{y})\leq f(\hat{x},y),
		\quad\forall y\in C\}$,
		\item for each $d\in\reals^2$, a point $\psi(d)\coloneqq \min\{\innerProd{d}{\nabla_xf(\hat{x},y)}:y\in Y\}$.
	\end{itemize}
	Then $\phi$ is locally Lipschitz continuous and directionally
	differentiable, and
	\[
	\frac{1}{2}
	\begin{bmatrix}
	\psi(1,0) - \psi(-1,0) \\ \psi(0,1) - \psi(0,-1)
	\end{bmatrix}
	\]
	is an element of $\genJac{\phi}(\hat{x})$.
\end{proposition}
\begin{proof}
	The optimal-value function $\phi$ has already been established to be locally Lipschitz
	continuous~\cite[Theorem~2.1]{DempMordZem} and directionally
	differentiable~\cite{Danskin},
	with directional derivatives given by $\dd{\phi}{\hat{x}}{d}
	= \psi(d)$ for each $d\in\reals^2$. The claimed result  then
	follows immediately from {\cref{thm:nonconvex}}.
\end{proof}

Observe that, unlike several established sensitivity results for
optimal-value
functions~\cite{BonnansShapiroReview,BonnansShapiroBook}, the above
result does not require second-order sufficient optimality conditions
to hold, and does not require unique solutions of the optimization
problems defining~$\phi$.
An analogous approach describes subgradients of the
Tsoukalas-Mitsos convex relaxations~\cite{Tsoukalas} of composite
functions of two variables; the Tsoukalas-Mitsos approach is based
entirely on analogous
optimal-value functions.

\section{Conclusion}
\label{sec:conclusion}

{For a bivariate nonsmooth function under minimal
	assumptions, the compass difference introduced in this article is guaranteed
	to be a
	subgradient and may be computed using four calls to a
	directional-derivative evaluation oracle. This remains true for
	nonconvex functions, with the ``subgradient'' understood in this case to
	be an element of Clarke's generalized gradient. Thus, for such
	functions, centered finite differences will necessarily converge to
	a subgradient as the perturbation width tends to zero.} The
presented examples  show that this new relationship between directional
derivatives and subgradients may be useful for functions of two
variables, {and may in some cases provide the only known way to
	evaluate a subgradient,} but does not extend directly to functions of three or more
variables. Such a nontrivial extension represents a possible avenue
for future work.


\bibliographystyle{jnsao}
\bibliography{references}

\end{document}